\documentclass[12pt]{article}
\usepackage[a4paper,margin=2.6cm,top=2.6cm,bottom=3cm]{geometry}
\usepackage[T1]{fontenc}
\usepackage{mathpazo}
\usepackage[scaled=0.92]{helvet}
\usepackage{amsmath,amssymb,amsthm,graphicx,booktabs}
\usepackage[final]{microtype}
\usepackage{xcolor}
\definecolor{accent}{RGB}{24,40,72}
\definecolor{gold}{RGB}{242,183,5}
\definecolor{cream}{RGB}{255,252,240}
\definecolor{muted}{RGB}{100,100,96}
\usepackage{enumitem}
\setlist{nosep, leftmargin=14pt}
\usepackage[font=small,labelfont={sf,bf,color=accent},labelsep=period]{caption}
\usepackage{titlesec}
\titleformat{\section}{\normalfont\sffamily\large\bfseries\color{accent}}{\thesection}{0.8em}{}
\titleformat{\subsection}{\normalfont\sffamily\normalsize\bfseries\color{accent}}{\thesubsection}{0.8em}{}
\titlespacing*{\section}{0pt}{2.2ex plus .5ex minus .2ex}{1.0ex plus .2ex}
\titlespacing*{\subsection}{0pt}{1.8ex plus .4ex minus .2ex}{0.8ex plus .2ex}
\usepackage{fancyhdr}
\fancypagestyle{plain}{\fancyhf{}\fancyfoot[C]{\footnotesize\sffamily\color{muted}\thepage}}
\usepackage[colorlinks,linkcolor=accent,citecolor=accent,urlcolor=accent,
  pdftitle={A theoretically exact filtered back-propagation formula for single-time wave inversion},
  pdfauthor={Markus Haltmeier}]{hyperref}

\newtheoremstyle{thm}{0pt}{0pt}{\itshape}{}{\sffamily\bfseries\color{accent}}{.}{0.5em}{}
\newtheoremstyle{rem}{0pt}{0pt}{\normalfont}{}{\sffamily\bfseries\color{accent}}{.}{0.5em}{}
\theoremstyle{thm}
\newtheorem{theorem}{Theorem}
\newtheorem{proposition}[theorem]{Proposition}
\newtheorem{lemma}[theorem]{Lemma}
\newtheorem{corollary}[theorem]{Corollary}
\theoremstyle{rem}
\newtheorem{remark}[theorem]{Remark}

\usepackage[most]{tcolorbox}
\tcbset{thmbox/.style={enhanced,breakable,sharp corners,boxrule=0pt,frame hidden,
  colback=cream,borderline west={2.5pt}{0pt}{gold},
  left=8pt,right=8pt,top=5pt,bottom=5pt,before skip=1.6ex,after skip=1.6ex}}
\tcolorboxenvironment{theorem}{thmbox}
\tcolorboxenvironment{proposition}{thmbox}
\tcolorboxenvironment{lemma}{thmbox}
\tcolorboxenvironment{corollary}{thmbox}
\tcolorboxenvironment{remark}{thmbox,colback=white,borderline west={1.2pt}{0pt}{gold}}

\newcommand{\R}{\mathbb{R}}
\newcommand{\Rad}{\mathcal{R}}
\newcommand{\W}{\mathcal{W}}
\newcommand{\F}{\mathcal{F}}
\newcommand{\A}{\mathcal{A}}
\newcommand{\Ht}{\mathcal{H}}
\newcommand{\Si}{\operatorname{Si}}
\newcommand{\supp}{\operatorname{supp}}

\renewenvironment{abstract}{\small\noindent\begin{minipage}{\linewidth}%
  {\sffamily\bfseries\color{accent}Abstract.}\ }{\end{minipage}}
\newenvironment{keywords}{\par\vspace{1.2ex}\small\noindent{\sffamily\bfseries\color{accent}Keywords.}\ }{\par}

\begin{document}
\thispagestyle{plain}
\begin{center}
{\sffamily\bfseries\LARGE\color{accent} A theoretically exact filtered back-propagation\\[0.3ex] formula for single-time wave inversion\par}
\vspace{2.2ex}
{\large Markus Haltmeier\par}
\vspace{0.8ex}
{\small Department of Mathematics, University of Innsbruck\\
Technikerstra\ss e 13, 6020 Innsbruck, Austria\\
\href{mailto:markus.haltmeier@uibk.ac.at}{\texttt{markus.haltmeier@uibk.ac.at}}\par}
\vspace{1.2ex}
{\small\sffamily\color{muted} Submitted conference paper (4-page proceedings version)\par}
\vspace{1ex}
{\color{gold}\rule{0.35\linewidth}{1.2pt}}
\end{center}
\vspace{1ex}

\begin{abstract}
We consider single-time wave inversion, that is, recovering the initial
pressure $f$ from the wave field $\W_Tf = u(\cdot,T)$ at a single time $T$,
where $u$ solves the wave equation with sound speed $c$ and initial data
$(f,0)$. In two dimensions this is the main reconstruction step of full-field photoacoustic projection imaging, where each camera image is a two-dimensional wave field at a single time instant. An efficient direct method is back-propagation by
the Fourier multiplier $2\cos(cT|\xi|)$. It yields exact inversion in odd
dimensions but not in two. We give its error in closed form, with bounds, and
derive an exact explicit filtered back-propagation $f = 2\W_T\mathcal M_K\W_Tf$ in analogy
to filtered backprojection for the Radon transform.
\end{abstract}

\begin{keywords}
Full-field detection, wave equation, reconstruction formula, photoacoustic tomography

\par\vspace{0.6ex}\noindent{\sffamily\bfseries\color{accent}MSC 2020.}\ 35L05, 35R30, 44A12, 92C55
\end{keywords}
\vspace{2ex}

\section{Introduction}\label{sec:intro}

In camera-based full-field photoacoustic tomography (PAT), an optical phase-contrast
system and a camera record projections of the 3D acoustic
pressure field at a single time $T$ after the laser pulse
\cite{nuster2010}, instead of time-resolved signals on
a detection surface as in standard PAT
\cite{poudel2019survey}. Because the X-ray projection
commutes with the Laplacian, every camera image is itself a two-dimensional
wave field whose initial value is the linear projection of the initial pressure
\cite{nuster2010}. Reconstruction thus splits into a 2D
single-time wave inversion per projection, followed by a Radon inversion
across projections. This is how camera-based PAT is processed in
practice \cite{nuster2019,nuster2021}. It is modular and fast: every camera
image is processed independently at the cost of a few FFTs. It also lets incomplete images be completed in 2D first before the final image formation step. Note that the reduction to 2D uses a constant speed of sound; for variable sound speed, full 3D methods are a natural alternative \cite{zangerl2019,nguyen2022}.

\textbf{Main contributions.} (i) We show that Fourier back-propagation,
while exact in odd dimensions (Prop.~\ref{prop:odd}), leaves a smooth bias in
2D, which we give in closed form (Prop.~\ref{prop:bias}). (ii) For every
$cT>R$ ($R$ the source radius) we derive an explicit radial filter $\mathcal M_K$ that
gives the exact inversion formula $f = 2\W_T\mathcal M_K\W_Tf$. The formula is expressed
explicitly in physical space by circular means (Thm.~\ref{thm:filter},
Cors.~\ref{cor:trunc}--\ref{cor:phys}), and all such filters within this
class are characterised. This results in a formula similar to the well-known
filtered backprojection (FBP) formulas for standard X-ray CT
\cite{natterer2001} as well as for standard PAT
\cite{haltmeier2014universal}.
(iii) Simulations compare back-propagation, filtered back-propagation and
Radon-domain band selection in accuracy, noise and run time.

\section{Theory}\label{sec:fdb}

\subsection{Problem  statement} 

Let $f\in C_c^\infty(\R^d)$ with $\supp f\subseteq \overline{B_d(0,R)}$, and let $u$ solve $\partial_t^2u = c^2 \Delta u$
with $u(\cdot,0)=f$, $\partial_t u(\cdot,0)=0$ where $c$ is the constant speed of sound. Single-time wave inversion is the problem of recovering $f$ from
\begin{equation}\label{eq:model}
g=\W_Tf\triangleq u(\cdot,T)=\cos\bigl(cT\sqrt{-\Delta}\bigr)f  \,,
\end{equation}
where $cT>R$; see \cite{zangerl2019,nguyen2022} for uniqueness and stability. 
With respect to the $L^2(\R^d)$ inner product, $\W_T$
is self-adjoint and $\W_Tg$ is the time reversal of the data
$(g,0)$ from $T$ to $0$.

\begin{remark}[Relevance for PAT]
For full-field PAT, $d=2$ and $f = X_\theta f_{\textnormal{3D}}$ is the
projection of the 3D initial pressure in direction $\theta$; if
$\supp f_{\textnormal{3D}}\subseteq\overline{B_3(0,R)}$ then
$\supp f\subseteq\overline{B_2(0,R)}$ \cite{nuster2010}. We assume complete
images; in practice the sample shadows a central disc, and data outside the
shadow are treated in separate work. Besides Radon-domain band selection
\cite{nuster2010} and iterative methods \cite{zangerl2019,nguyen2022}, Fourier
back-propagation \eqref{eq:fdb} is used in practice \cite{nuster2019,nuster2021};
its error is characterised analytically in this work.
\end{remark}

\subsection{Notation} 

We denote by $\hat h(\xi) = \F [h](\xi) \triangleq \int_{\R^d} h(x)\, e^{-i x\cdot\xi}\, \mathrm{d}x$ the Fourier transform of $h$ on $\R^d$.  For a radial symbol $m$, the Fourier multiplier $\mathcal M[m]$ acts by $\F [ \mathcal M[m] h ] (\xi) \triangleq m(|\xi|) \F h(\xi)$; such multipliers commute with each other. In particular, $\W_T=\mathcal M[\cos(cT\rho)]$.

The Radon transform of a function $h$ on $\R^d$ is defined by
$\Rad h(\varphi,s) \triangleq \int_{x\cdot\varphi=s}h(x)\, \mathrm{d}x$ for
$\varphi\in S^{d-1}$ and $s\in\R$. For $d=2$, the FBP
inversion formula for the Radon transform reads
$h=\Rad^*\Lambda\Rad h /(4\pi)$, where
$\Rad^*G(x)=\int_{S^1}G(\varphi,x\cdot\varphi)\,d\varphi$ is the $L^2$-adjoint of
$\Rad$, known as the backprojection operator, and $\Lambda = \Ht_s\partial_s$
is the ramp filter with symbol $|\sigma|$; here $\Ht_s$ is the Hilbert
transform in $s$. We write $\Lambda G=\lambda*G$ with the ramp kernel
$\lambda\triangleq\Lambda\delta$, $\lambda(s)=-1/(\pi s^2)$ for $s\neq0$
\cite{natterer2001}. More generally, for an even kernel $K$ on
$\R$, the Fourier slice theorem gives
\begin{equation}\label{eq:slice}
\Rad^*(K*\Rad h) =4\pi \, \mathcal M_K h \,,
\qquad
\mathcal M_K\triangleq \mathcal M\bigl[\widehat K(\rho)/\rho\bigr] \,.
\end{equation}
For $K=\lambda$ we have $\hat\lambda(\sigma)=|\sigma|$, hence
$\mathcal M_{\lambda}=I$ and \eqref{eq:slice} reduces to the FBP formula.

\subsection{Back-propagation and its error}

Fourier back-propagation (BP) \cite{nuster2019,nuster2021} is
defined by \begin{equation}\label{eq:fdb}
f_{\rm BP}
\triangleq 2\W_T g
=
\F^{-1}\bigl[2\cos(cT|\xi|)\,\hat g\,\bigr] \,,
\end{equation}
where $g=\W_Tf$, and is often used for image formation in full-field PAT.

\begin{lemma}[Auxiliary wave identities]\label{lem:basic}
(a) $2\W_T\W_T=I+\W_{2T}$, hence $f_{\rm BP}=f+\W_{2T}f$.
(b) $\Rad\W_th(\varphi,s)=(\Rad h(\varphi,s+ct)+\Rad h(\varphi,s-ct))/2$.
\end{lemma}

\begin{proof}
Item (a) follows from $2\cos^2(cT\rho)=1+\cos(2cT\rho)$. For (b), by the
Fourier slice theorem the 1D Fourier transform of $\Rad \W_t h(\varphi,\cdot)$ is
$\cos(ct\sigma)\hat h(\sigma\varphi)$, and multiplication by
$\cos(ct\sigma)=(e^{ict\sigma}+e^{-ict\sigma})/2$ averages the shifts by $\pm ct$.
\end{proof}

\begin{proposition}[BP is exact in odd dimensions]\label{prop:odd}
If $d$ is odd and $cT\ge R$, then $f_{\rm BP}=2\W_T\W_Tf=f$ on $B_d(0,R)$. For $d=3$,
this gives the explicit FBP-type inversion formula
\begin{align} \label{eq:inv-3d-1}
f(x)
& =2\W_T\W_Tf (x)
\\ \label{eq:inv-3d-3}
f(x)& = \frac1{2\pi} \int_{S^2} \bigl( g + cT\,\theta\cdot\nabla g \bigr)(x+cT\theta)\,d\theta
\end{align}
\end{proposition}

\begin{proof}
By Lemma~\ref{lem:basic}(a) it suffices that $\W_{2T}f=0$ on $B_d(0,R)$. For odd
$d$, $\W_tf(x)$ depends only on $f$ near the sphere $\partial B_d(x,ct)$ (strong
Huygens principle for $d\ge3$, d'Alembert's formula
$\W_t f (s) = f(s-ct)/2+f(s+ct)/2$ for $d=1$). Hence
$\W_{2T}f(x)=0$ whenever $\partial B_d(x,2cT)\cap\overline{B_d(0,R)}=\emptyset$,
in particular for $|x|<2cT-R$, and $2cT-R\ge R$. Formula \eqref{eq:inv-3d-3} follows from Kirchhoff's formula
$\W_Tg(x)= \partial_r[r\int_{S^2}g(x+r\theta)\,d\theta/(4\pi )]_{r=cT} $
\cite{john1982} by differentiating under the integral.
\end{proof}

In two dimensions the strong Huygens principle fails and thus 
$\W_{2T}f \neq0$ in general.

\begin{proposition}[BP error in 2D]\label{prop:bias}
Let $d=2$ and $cT>R$. Then for all $x\in B_d(0,R)$ we have 
\begin{equation}\label{eq:bias}
E_{\rm BP}(x) = f_{\rm BP}(x)-f(x)
=-\frac{cT}{\pi}\int_{B_d(0,R)}
\frac{f(y)\,dy}{\bigl(4c^2T^2-|x-y|^2\bigr)^{3/2}} .
\end{equation}
Moreover, $|E_{\rm BP}(x)|\le \| f\|_{1}  cT / (8\pi(c^2T^2-R^2)^{3/2})$, and if $f\ge0$, $f\not\equiv0$, then $f(x)- f_{\rm BP}(x)$ is positive with lower bound $\|f\|_{1}/(8\pi c^2T^2)$.\end{proposition}

\begin{proof}
According to Poisson's formula \cite{john1982}, $\W_t f(x)$ equals 
\[\partial_t \int_{B_d(x,ct)} f(y) (c^2t^2-|x-y|^2)^{-1/2}/(2\pi c) dy .\]
For $x\in B_d(0,R)$ and $ct > 2R$ the ball $B_d(x,ct)$ contains $\overline{B_d(0,R)}$, so the domain of integration is fixed and we may differentiate under the
integral; with Lemma~\ref{lem:basic}(a) this gives $\W_{2T}f(x)$. The bounds follow from
$0\le|x-y|<2R$.
\end{proof}

The kernel in \eqref{eq:bias} is smooth on $|x-y|\le2R$, so the error is a
$C^\infty$ function on $B_d(0,R)$: back-propagation produces no edge artefacts, but
a low-frequency bias $E_{\rm BP}(x)$, bounded by multiples of the total mass $\| f\|_{1}$. For
$f\ge0$ and $cT \gg  R$ it is $E_{\rm BP}(x) \approx  -\| f\|_{1} /(8\pi c^2T^2)$. For a homogeneous disc $f=\chi_{B_d(0,R)}$ the relative error at $x=0$ is $2cT/\sqrt{4c^2T^2-R^2}-1$, i.e.\ $9.1\%$ for $cT=1.25R$ and $3.3\%$ for $cT=2R$.

\subsection{Exact filtered back-propagation in 2D}\label{sec:exact}

Now we derive the exact filtered back-propagation formula for $d=2$. Since $\W_T^*=\W_T$,
back-propagation $f_{\rm BP}=2\W_T^*\W_Tf$ is twice the normal operator applied
to $f$. By Lemma~\ref{lem:basic}, in the Radon domain it becomes
$\Rad f_{\rm BP}=\A_T\Rad f$ with
\begin{equation}\label{eq:AT}
\A_T G(s)\triangleq G(s)+ G(s-2cT)/2+G(s+2cT)/2 \,,
\end{equation}
that is, $\A_T=2\W^{(1)}_T\W^{(1)}_T$ is twice the normal operator of the 1D
wave operator $\W^{(1)}_T$ in each Radon  
direction. Thm.~\ref{thm:filter} requires that $\A_TK$ equals on $(-2R,2R)$
the ramp kernel $\lambda$, defined by $\hat\lambda(\sigma) =|\sigma|$. Explicit
kernels satisfying this are constructed in Cors.~\ref{cor:trunc} and
\ref{cor:series}.

\begin{theorem}[Exact filters]\label{thm:filter}
Let $cT>R$ and let $K$ be an even 1D distribution with $K-\lambda\in L^1(\R)$
and
\begin{equation}\label{eq:K}
\A_T K (u)= \lambda(u)
\; \text{ for } \; |u|<2R \,.
\end{equation}
Then for $x \in B_d(0,R)$ we have
\begin{align}
\label{eq:exact2}
f(x)&= \tfrac1{4\pi}\Rad^*\bigl(K*\Rad f_{\rm BP}\bigr)(x)
\\ \label{eq:exact3}
f(x)&= \mathcal M_K\,f_{\rm BP}(x)
\\  \label{eq:exact4}
f(x)&=\F^{-1}\Bigl[2\,\frac{\widehat K(|\xi|)}{|\xi|}\cos(cT|\xi|)\,\hat g\Bigr]
(x)
\\
\label{eq:exact1}
f(x)&=2\, (\W_T\mathcal M_K\W_T f) (x)
.
\end{align}
\end{theorem}

\begin{proof}
For $\eta\in\R$ let $(\tau_\eta G)(s)\triangleq G(s-\eta)$ denote the shift in
$s$, so that $\A_T = I + (\tau_{2cT} + \tau_{-2cT})/2$. We first derive
\eqref{eq:exact2} in four steps.
\emph{Step 1 (three bands):} By \eqref{eq:AT}, $\Rad f_{\rm BP} = \A_T\Rad f$ consists of the band $\Rad f$,
supported in $|s| \le R$, and two half copies shifted by $\pm2cT$.
\emph{Step 2 (effective kernel):} As shifts commute with convolution,
$K * \Rad f_{\rm BP} = (\A_TK) * \Rad f$,
that is, filtering the three bands with $K$ is the same as filtering the
single band $\Rad f$ with the effective kernel $\A_TK$.
\emph{Step 3 (only $(-2R,2R)$ matters):} Write $\A_TK=\lambda+D$. By \eqref{eq:K}, $D=0$ on $(-2R,2R)$. Since
$\supp\Rad f(\varphi,\cdot)\subseteq[-R,R]$, we have
$\supp(D*\Rad f(\varphi,\cdot))\subseteq\supp D+[-R,R]$, which does not meet
$(-R,R)$. Hence $K*\Rad f_{\rm BP}=\Lambda\Rad f$ for $|s|<R$.
\emph{Step 4 (backprojection):} For $x\in B_d(0,R)$ and $\varphi\in S^1$
we have $|x\cdot\varphi|<R$, so $\Rad^*G(x)$ uses $G$ only on $|s|<R$. By
Step 3 and the FBP formula,
$\Rad^*(K*\Rad f_{\rm BP})(x)= \Rad^*\Lambda\Rad f(x) = 4\pi f(x)$,
which gives \eqref{eq:exact2}.

We finally derive the remaining formulas from \eqref{eq:exact2}.
By \eqref{eq:slice} with $h=f_{\rm BP}$,
$\Rad^*(K*\Rad f_{\rm BP})= 4\pi \mathcal M_Kf_{\rm BP}$, which is
\eqref{eq:exact3}; with $\F f_{\rm BP}=2\cos(cT|\xi|)\F g$ this is
\eqref{eq:exact4}. Finally, $\mathcal M_K$ commutes with $\W_T$, so
$2\W_T\mathcal M_K\W_Tf=\mathcal M_K\,2\W_Tg=\mathcal M_Kf_{\rm BP}$, which is
\eqref{eq:exact1}.
\end{proof}

\begin{corollary}[Truncated ramp case $cT\ge2R$] \label{cor:trunc}
Let $cT\ge2R$ and $2R\le a\le2cT-2R$. Then $K_a \triangleq\chi_{[-a,a]}\lambda$
satisfies \eqref{eq:K}, and exact reconstruction formulas  \eqref{eq:exact2}--\eqref{eq:exact1} hold with
\begin{equation}\label{eq:Mb}
\frac{\widehat{K_a}(\rho)}{\rho}= \frac2\pi\Bigl[\Si(a\rho)+\frac{\cos(a\rho)}{a\rho}\Bigr]
\end{equation}
where $\Si(x)\triangleq\int_0^x\sin(t)/t\,\mathrm{d}t$ is the sine integral.
Equivalently, $f=f_{\rm BP}+\kappa_a*f_{\rm BP}$ on $B_d(0,R)$ with
\begin{equation}\label{eq:cb}
\kappa_a(z)=\frac{\sqrt{|z|^2-a^2}}{\pi^2\,a\,|z|^2}\,\chi_{\{|z|>a\}} .
\end{equation}
\end{corollary}

\begin{proof}
For $|u|<2R\le a$ we have $K_a(u)=\lambda(u)$ and $|u\mp2cT|>2cT-2R\ge a$,
so $K_a(u\mp2cT)=0$. Since $K_a=\lambda+\chi_{\{|s|>a\}}/(\pi s^2)$,
integration by parts gives
$\widehat{K_a}(\sigma)=\frac2\pi[\sigma\Si(a\sigma)+\cos(a\sigma)/a]$, which is
\eqref{eq:Mb}. By \eqref{eq:slice}, $\lambda$ contributes the
identity, and $\chi_{\{|s|>a\}}/(\pi s^2)$ the kernel
$\frac1{4\pi}\int_{S^1}\chi_{\{|z\cdot\varphi|>a\}}(\pi(z\cdot\varphi)^2)^{-1}
d\varphi=\frac{\tan\theta_0}{\pi^2|z|^2}$ for $|z|>a$, where
$\cos\theta_0=a/|z|$; this is \eqref{eq:cb}.
\end{proof}

Thus, for $cT\ge2R$, exact reconstruction is ordinary FBP
of $f_{\rm BP}$ with the ramp kernel cut off at $|s|=a$; the cut-off removes
the interaction with the two shifted bands. Since $\kappa_a=0$ on $B_d(0,a)$ and
$a\ge2R$, the correction $\kappa_a*f_{\rm BP}$ on $B_d(0,R)$ uses $f_{\rm BP}$ only
outside $B_d(0,R)$, where $f_{\rm BP}=\W_{2T}f$: the filter predicts the interior
tail of \eqref{eq:bias} from the outgoing part of the field.
$\widehat{K_a}(\rho)/\rho\to1$ as $\rho\to\infty$ and $\widehat{K_a}(\rho)/\rho\sim2/(\pi a\rho)$ as
$\rho\to0$; the coefficient of this low-frequency gain is smallest for
$a=2cT-2R$.

\begin{corollary}[All $cT>R$]\label{cor:series}
Let $R<cT$, $q(u)\triangleq\frac1{2\pi}[(u-2cT)^{-2}+(u+2cT)^{-2}]$ and
$S\triangleq\{2cT-2R<|u|<2R\}$ (empty if $cT\ge2R$). Define $p=0$ for $|u|\ge2R$,
$p=q$ for $|u|<2R$, $u\notin S$, and
$p(u)=\frac43[q(u)-\frac12q(u\mp2cT)]$ for $u\in S$, $\pm u>0$. Then
$K=\lambda+p$ satisfies \eqref{eq:K}, and \eqref{eq:exact2}--\eqref{eq:exact1} hold with
$\widehat K(\rho)/\rho=1+\hat p(\rho)/\rho$.
\end{corollary}
\begin{proof}
For $K=\lambda+p$ with $\supp p\subseteq[-2R,2R]$, \eqref{eq:K} is
$(I+P)p=q$ on $(-2R,2R)$ with
$Pp(u)=\frac12[p(u-2cT)+p(u+2cT)]$, because
$-\frac12\lambda(u\mp2cT)$ sum to $q$. For $|u|<2R$ at most one shift
$u\mp2cT$ stays in $(-2R,2R)$, and only if $u\in S$; hence $Pp=0$ off $S$, and
on $S$ two shifts return to $u$, so $P^2=\frac14\chi_S$. Therefore
$(I+P)(I-P)=I-\frac14\chi_S$ and $p=(I-P)\bigl(q/(1-\frac14\chi_S)\bigr)$,
which is the stated formula.
\end{proof}

\begin{corollary}[Filtered back-propagation by circular means]\label{cor:phys}
Let $M_h(x,r)\triangleq\frac1{2\pi}\int_{S^1}h(x+r\theta)\,d\theta$,
$cT\ge2R$ and $2R\le a\le2cT-2R$. Then on $B_d(0,R)$ we have
\begin{equation}\label{eq:phys}
f(x)=\frac2{cT}\int_0^{cT}\frac{r}{\sqrt{c^2T^2-r^2}}\,
\frac{\partial}{\partial r}\bigl[r\,M_h(x,r)\bigr]\,dr ,
\end{equation}
where, with $K_a$ from Cor.~\ref{cor:trunc},
\[
h\triangleq g+\frac{2}{\pi a}\int_a^\infty\frac{\sqrt{\rho^2-a^2}}{\rho}\,
M_g(\cdot,\rho)\,d\rho=\mathcal M_{K_a}g .
\]
For general $cT>R$, \eqref{eq:phys} holds with
\[
h\triangleq g+2\pi\int_0^\infty\rho\,\kappa(\rho)\,M_g(\cdot,\rho)\,d\rho=\mathcal M_Kg ,
\]
where $K=\lambda+p$ is from Cor.~\ref{cor:series} and the radial kernel
$\kappa\triangleq\frac1{4\pi}\Rad^*p$ of $\mathcal M_K-I$ is
\[\kappa(\rho)=\frac1\pi\int_0^{\min(\rho,2R)}\frac{p(s)\,ds}{\sqrt{\rho^2-s^2}} .\]
\end{corollary}
\begin{proof}
For a radial kernel $\kappa$, polar coordinates give
$\kappa*g=2\pi\int_0^\infty\rho\,\kappa(\rho)\,M_g(\cdot,\rho)\,d\rho$; with $\kappa=\kappa_a$ from
\eqref{eq:cb} this gives $h=g+\kappa_a*g=\mathcal M_{K_a}g$; for $K=\lambda+p$, \eqref{eq:slice} and $\Rad^*(p*\Rad g)=(\Rad^*p)*g$ give $h=g+\kappa*g=\mathcal M_Kg$. Poisson's formula \cite{john1982}
reads $\W_Th(x)=\frac1c\partial_t\bigl[\int_0^{ct} r\,M_h(x,r)(c^2t^2-r^2)^{-1/2}dr\bigr]_{t=T}$;
substituting $r=cts$ and differentiating under the integral gives
$\W_Th(x)=\frac1{cT}\int_0^{cT} r(c^2T^2-r^2)^{-1/2}\partial_r[r\,M_h(x,r)]\,dr$. Finally
$f=2\W_T\mathcal M_K\W_Tf=2\W_Th$ on $B_d(0,R)$ by Thm.~\ref{thm:filter}. Since
$\supp g\subseteq \overline{B_d(0,R+cT)}$, only $\rho\le2R+2cT$ contributes.
\end{proof}

\begin{remark}
Within the class of Thm.~\ref{thm:filter}, \eqref{eq:K} is also necessary:
$2\W_T\mathcal M_K\W_T-I$ has the radial kernel
$k(r)=\frac1\pi\int_0^r(\A_TK-\lambda)(s)(r^2-s^2)^{-1/2}ds$ for $r<2R$, which vanishes on
$B_d(0,2R)$ if and only if $\A_TK=\lambda$ on $(-2R,2R)$, by injectivity of the
Abel transform. Like the band selection of \cite{nuster2010}, the filtered
back-propagation is exact on $B_d(0,R)$, but it needs no Radon transform of the
camera image, only circular means or FFTs. For $cT\le R$, writing $K=\lambda+p$,
\eqref{eq:K} would require the integrable function $(I+P)p$ to equal $q$ on
$(-2R,2R)$, which is not integrable near $\pm2cT\in[-2R,2R]$; hence no filter
of this class exists.
\end{remark}

\section{Numerical results}\label{sec:num}

We use $R=c=1$ and grid spacing $h=0.02$ on $[-L,L]^2$ with $L=2R+2cT+0.5$
(no periodic wrap-around on $B_d(0,R)$). Data are simulated on a
three times finer grid and averaged over the cells. Phantoms are a smoothed disc
of radius $0.85R$ and a vessel-like tree in $B_d(0,0.8R)$ (Fig.~\ref{fig:num}a);
errors are relative $L^2$ errors on $B_d(0,0.98R)$. $\W_T$ is applied by FFT.
Filtered back-propagation computes $f_{\rm BP}+\kappa*f_{\rm BP}$ by FFT
convolution with $\kappa$ from Cor.~\ref{cor:phys}, cell-averaged and cut off at
$|z|=2R+2cT$, which is exact on $B_d(0,R)$ ($\kappa_a$, $a=2cT-2R$, gives the
same errors for $cT\ge2R$).

\begin{figure}[t]
\centering
\includegraphics[width=0.62\linewidth]{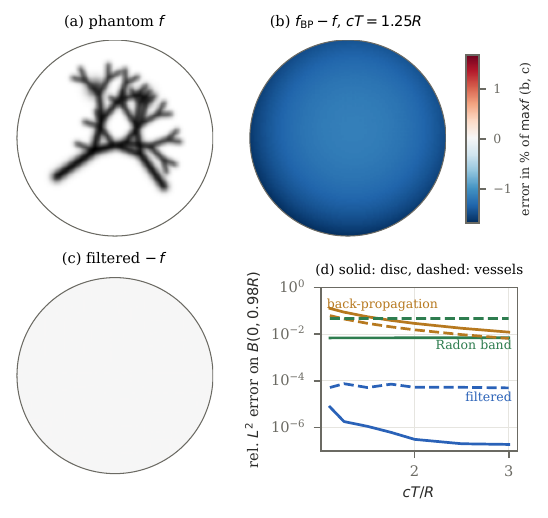}
\caption{(a) Vessel phantom. (b), (c) Errors of back-propagation and of
filtered back-propagation at $cT=1.25R$, same colour scale; the error in (c) is below $0.02\,\%$ and thus invisible. (d) Relative
$L^2$ error on $B_d(0,0.98R)$ versus $cT/R$.}\label{fig:num}
\end{figure}

Back-propagation has errors of $0.7$--$12.7\,\%$ for $cT/R=3,\dots,1.1$: a smooth
negative bias (Fig.~\ref{fig:num}b) that agrees with \eqref{eq:bias} up to
$1.3\cdot10^{-4}$. Filtered back-propagation reduces the error to at most
$7.5\cdot10^{-5}$ (Fig.~\ref{fig:num}c,d); for $h=0.01$ it drops to $10^{-6}$, so
the remainder is discretisation error, whereas the back-propagation error does
not change. Radon-domain band selection \cite{nuster2010} (scikit-image, 720
angles) gives $4.7\,\%$ for the vessels, which equals the discretisation error of
$\Rad^{-1}\Rad f$. With $2\,\%$ noise, back-propagation and filtered
back-propagation amplify white noise equally (gain $1.41$--$1.44$), since the
filter changes only low frequencies. Per image, filtered back-propagation takes $0.015$\,s, the Radon route $1.4$\,s.

\section{Conclusion}\label{sec:concl}

Back-propagation of single-time wave data is exact in odd dimensions but
biased in two, the dimension of full-field PAT projection images.
An explicit radial filter turns it into an exact filtered back-propagation
for every $cT>R$, at the cost of back-propagation with a modified Fourier
multiplier. Open questions are the
optimal choice among the exact filters under noise, the case $cT\le R$, and
the use of the filter as a preconditioner for variable sound speed and
attenuation \cite{zangerl2019}.

\section{Compliance with ethical standards}\label{sec:ethics}
This is a theoretical and numerical simulation study for which no ethical
approval was required.

\section{Acknowledgments}\label{sec:ack}
No funding was received; the author has no competing interests.
The research question, approach and main results were conceived by the author.
Claude (Anthropic, Opus 5.5) assisted in working out proof details (Sec.~2),
drafting text (all sections) and drafting the simulation code (Sec.~3). The
author verified all results and takes full responsibility for the content.

\bibliographystyle{IEEEbib}
\bibliography{refs}

\begin{thebibliography}{1}

\bibitem{nuster2010}
R.~Nuster, G.~Zangerl, M.~Haltmeier, and G.~Paltauf,
\newblock ``Full field detection in photoacoustic tomography,''
\newblock {\em Opt. Express}, vol. 18, no. 6, pp. 6288--6299, 2010.

\bibitem{poudel2019survey}
J.~Poudel, Y.~Lou, and M.~A. Anastasio,
\newblock ``A survey of computational frameworks for solving the acoustic
  inverse problem in three-dimensional photoacoustic computed tomography,''
\newblock {\em Phys. Med. Biol.}, vol. 64, no. 14, pp. 14TR01, 2019.

\bibitem{nuster2019}
R.~Nuster and G.~Paltauf,
\newblock ``Optimization of image quality in photoacoustic tomography using
  spatial projection data,''
\newblock in {\em Proc. SPIE 10878, Photons Plus Ultrasound: Imaging and
  Sensing}, 2019, p. 108784N.

\bibitem{nuster2021}
R.~Nuster, D.~Brandmueller, and G.~Paltauf,
\newblock ``Benchtop photoacoustic tomograph with camera based ultrasound
  detection,''
\newblock in {\em Proc. SPIE-OSA 11923, European Conferences on Biomedical
  Optics}, 2021, p. 119230L.

\bibitem{zangerl2019}
G.~Zangerl, M.~Haltmeier, L.~V. Nguyen, and R.~Nuster,
\newblock ``Full field inversion in photoacoustic tomography with variable
  sound speed,''
\newblock {\em Appl. Sci.}, vol. 9, no. 8, pp. 1563, 2019.

\bibitem{nguyen2022}
L.~V. Nguyen, M.~Haltmeier, R.~Kowar, and N.~Do,
\newblock ``Analysis for full field photoacoustic tomography with variable
  sound speed,''
\newblock {\em SIAM J. Imaging Sci.}, vol. 15, no. 3, pp. 1213--1228, 2022.

\bibitem{natterer2001}
F.~Natterer,
\newblock {\em The Mathematics of Computerized Tomography}, vol.~32 of {\em
  Classics in Applied Mathematics},
\newblock SIAM, Philadelphia, 2001.

\bibitem{haltmeier2014universal}
M.~Haltmeier,
\newblock ``Universal inversion formulas for recovering a function from
  spherical means,''
\newblock {\em SIAM J. Math. Anal.}, vol. 46, no. 1, pp. 214--232, 2014.

\bibitem{john1982}
F.~John,
\newblock {\em Partial Differential Equations},
\newblock Springer, New York, 4th edition, 1982.

\end{thebibliography}

\end{document}